\pdfoutput=1                         
\documentclass{article}

\usepackage{geometry}
\usepackage[T1]{fontenc}

\usepackage{amsmath}          
\usepackage{amssymb,amsfonts}
\usepackage{mathrsfs}         
\usepackage{mathtools}        
\usepackage{stmaryrd}         
\usepackage{bm}
\usepackage{comment}          

\usepackage{graphicx}
\usepackage{float}
\usepackage{standalone}
\usepackage{subcaption}
\usepackage{booktabs}
\usepackage{colortbl}
\usepackage{siunitx}
\usepackage[backend=bibtex, giveninits=true]{biblatex}

\usepackage{pgfplots}
\usepackage{pgfplotstable}
\pgfplotsset{compat=newest}
\usepackage{tikz}
\usetikzlibrary{fadings}
\usetikzlibrary{patterns}
\usetikzlibrary{shadows.blur}
\usetikzlibrary{shapes}
\usetikzlibrary{calc}
\usetikzlibrary{positioning}
\usetikzlibrary{hobby}
\usetikzlibrary{patterns.meta, arrows.meta}
\usetikzlibrary{backgrounds}

\usepackage{amsthm}           

\usepackage{hyperref}
\usepackage{cleveref}

\newtheorem{theorem}{Theorem}[section]

\newtheorem{proposition}[theorem]{Proposition}

\newtheorem{rem}{Remark}[section]

\crefname{figure}{Figure}{Figures}\Crefname{figure}{Figure}{Figures}
\crefname{lemma}{Lemma}{Lemmas}\Crefname{lemma}{Lemma}{Lemmas}
\crefname{proposition}{Proposition}{Propositions}\Crefname{proposition}{Proposition}{Propositions}
\crefname{rem}{Remark}{Remarks}\Crefname{rem}{Remark}{Remarks}
\crefname{assum}{Assumption}{Assumptions}\Crefname{assum}{Assumption}{Assumptions}

\usepackage{macros}
\title{Analysis-aware interface coarsening for elliptic problems driven by an equilibrated flux estimator}
\author{%
  Denise Grappein\thanks{MOX, Department of Mathematics, Politecnico di Milano, Italy
    (\texttt{denise.grappein@polimi.it}). Member of the GNCS INdAM group.}
  \and
  Philipp Weder\thanks{Chair for Numerical Modelling and Simulation,
    \'Ecole Polytechnique F\'ed\'erale de Lausanne (EPFL), Switzerland
    (\texttt{philipp.weder@epfl.ch}).}%
}
\date{\today}

\begin{document}
\maketitle

\begin{abstract}
  We propose an \emph{a posteriori} error estimator that relies on an equilibrated flux reconstruction to enable analysis-aware interface coarsening decisions in elliptic interface problems. Interface coarsening consists of simplifying an internal interface that separates regions characterized by different physical properties in order to facilitate the meshing process and reduce the total number of degrees of freedom. 
This process extends the concept of \emph{defeaturing}, where small features on the boundary of a computational domain are removed.
Here, the features lie on an interior interface instead. The focus is on a Laplace problem with
a discontinuous diffusion coefficient. 
The estimator accounts for both the modeling error arising from the interface coarsening and the numerical error from the discrete approximation of the solution to the coarsened-interface problem.
It is localized on the mesh elements, and its constants explicitly track the contrast between the diffusion coefficients. The impact of different features can be assessed individually, yielding local contributions that indicate which features to coarsen and which to retain for a given mesh size.
The use of an equilibrated flux reconstruction allows us to sharply bound the bulk numerical source of error. We prove the reliability of the estimator and verify it across several numerical examples, including the case where an internal interface is fully removed.
\end{abstract}

\section{Introduction}

The geometric models appearing in engineering applications are often highly detailed, including features spanning multiple scales. This level of detail significantly increases the computational cost of numerical simulations by complicating mesh generation and increasing the number of degrees of freedom.
The process of removing, prior to simulation, geometric features considered irrelevant to the accuracy of the solution of a given PDE is called \emph{defeaturing}.

While defeaturing has historically been approached by exploiting \emph{a priori} knowledge of the geometry and simulated physical processes \cite{fine_automated_2000, foucault_mechanical_2004, thakur_survey_2009}, automated approaches require a quantitative, certified criterion to assess whether a given feature can be safely removed. This requirement has motivated the development of \emph{a posteriori} defeaturing error estimators, capable of evaluating the impact of feature removal directly on the simplified geometry, without requiring the solution of the fully detailed problem. 
In \cite{buffa_analysis-aware_2022}, a rigorous framework for analysis-aware defeaturing of the Poisson problem was introduced to address the case of \emph{boundary defeaturing}. More precisely, it covers the removal of features on the exterior boundary subject to Neumann boundary conditions. The Dirichlet case is treated in \cite{weder_analysis-aware_2025}, and \cite{weder_certified_2025} extends the estimates to quantities of interest other than the energy norm.
Such defeaturing error estimators account for the modeling error introduced by neglecting each feature and can be combined with a classical \emph{a posteriori} estimator of the discretization error to assess the relative importance of the two error contributions. In \cite{ buffa_adaptive_2024}, this was done using a residual-based estimator, whereas in \cite{buffa_equilibrated_2024}, an equilibrated flux error estimator was employed, with the advantage of yielding upper bounds free of unknown constants for the discretization error component. The resulting total-error estimators, accounting for both defeaturing and discretization errors, were used in \cite{antolin_analysisaware_2024, buffa_adaptive_2025} as part of a comprehensive adaptive strategy that allows for simultaneous mesh and geometry refinement.

In contrast to the works cited above that address cases in which the features lie on the domain's external boundary, the present work aims to exploit the same framework to simplify or remove interior interfaces separating regions with different diffusion coefficients. The defeaturing operation here consists of exchanging the coefficient within the feature with the one characterizing the neighboring material region. The goal remains to estimate the impact of this simplification on the solution's accuracy. To this end, we derive an equilibrated flux estimator reconstructed from the solution of the simplified problem, in which the constants involved in the upper bound explicitly depend on the contrast between the diffusion coefficients. The interface coarsening results in a simplified meshing procedure, as the mesh can be completely independent of the neglected features.

Several alternative strategies have been developed in the literature to handle problems involving interior material interfaces without resorting to interface-fitted meshes. A classical approach is to resort to homogenization (see, among others \cite{bakhvalov_homogenisation_1989, bensoussan_asymptotic_2011}), where a periodically heterogeneous medium is replaced by an effective homogeneous one. A different class of methods, broadly referred to as unfitted methods, maintains the physical domain or an approximation of it, allowing for almost arbitrary intersections with the surrounding mesh. Without aiming at proposing an exhaustive bibliography, we mention, for instance,  Immersed Boundary methods (see e.g.~\cite{peskin_immersed_2002}), Penalty methods (see e.g.~\cite{babuska_finite_1973}), Fictitious or Embedding  Domain methods (see e.g.~\cite{borgers_finite_1990, boffi_finite_2003}), the Cut Element method (see e.g.~\cite{burman_fictitious_2010,burman_fictitious_2012}) and the eXtended Finite Element method (see e.g.~\cite{fries_extendedgeneralized_2010}). Clearly, the use of unfitted approaches significantly simplifies mesh generation compared to standard conforming approaches, but imposing interface conditions requires extra work. 

The present work pursues a fundamentally different objective. Instead of seeking strategies to resolve the interface, we intentionally simplify the model by coarsening or removing it, and quantify the impact of this simplification by means of a computable \emph{a posteriori} error estimator. The estimator depends on the numerical solution of the coarsened interface problem and accounts for the contrast between the diffusion coefficients on either side of the interface. The intent is similar to the one pursued by Repin and coauthors in \cite{repin_combined_2012, weymuth_posteriori_2017}, with the main difference that we aim at providing a local quantifier of the modeling error introduced by the interface coarsening, allowing one to assess the impact of the coarsening of individual features of the original interface.  Such a local estimator could hence be employed in a pipeline similar to \cite{buffa_adaptive_2025} to adaptively coarsen or refine the interface without resorting to a predefined sequence of coarsened interfaces, as is instead the case in \cite{repin_combined_2012}. Furthermore, the estimator proposed herein is based on an equilibrated flux reconstructed from the numerical flux of the coarsened interface problem. The equilibrated flux is computed by a patch-wise equilibration procedure in which completely independent local Neumann problems are solved on elements sharing a node, thus making the computation of the equilibrated flux easily parallelizable.  

When the simplification results in the total removal of the interface, the computational mesh used to compute the defeatured solution and the corresponding equilibrated flux reconstruction can be completely independent of the interface itself. When a coarsened interface remains instead, we assume the mesh to conform to it. The intended target of the coarsening is hence an interface that admits a straightforward discretization, avoiding the need to resort, for example, to unfitted techniques for computing the discrete solution and the equilibrated flux. 

The rest of the manuscript is organized as follows: In \Cref{sec:notation and model problem}, we introduce the model transition problem with the exact and coarsened interfaces. In \Cref{sec:flux}, we recall the equilibrated flux reconstruction approach, introduce the combined error estimator, and prove its reliability. Several numerical experiments are presented in \Cref{sec:numerical experiments}, showcasing the estimator's reliability, efficiency, and robustness with respect to mesh and feature size variations as well as varying material contrasts. Finally, we offer our conclusions in \Cref{sec:conclusion}.
\section{Notation and model problem}
\label{sec:notation and model problem}
We begin this section by introducing the notation we will use throughout the paper.  Let $\omega \subset \mathbb{R}^d$, $d=\lbrace 2,3 \rbrace$, be any Lipschitz domain, i.e., an open and connected subset with Lipschitz-regular boundary. We denote by $|\omega|$ the measure of $\omega$, by $(\cdot,\cdot)_\omega$ the $L^2$ inner product on $\omega$, and by $||\cdot||_\omega$ the corresponding norm. We denote by $\partial \omega$ the boundary of $\omega$, and we write $\langle \cdot, \cdot\rangle_{\partial \omega}$ for the duality pairing on the boundary.
Consider a $d$-dimensional bounded Lipschitz domain $\domain \subset \R^d, d \in \{2, 3\}$, such that $\domain = \interior{\overline{\domaincomp{1}} \cup \overline{\domaincomp{2}}}$; see \cref{fig:model problem:exact domain}. We assume the components $\domaincomp{1}$ and $\domaincomp{2}$ to be bounded Lipschitz domains as well. In particular, the interface $\interface \coloneqq \overline{\domaincomp{1}} \cap \overline{\domaincomp{2}}$ is a Lipschitz boundary.
\begin{figure}
    \centering
    \begin{subfigure}[B]{0.42\textwidth}
    \centering
    \begin{tikzpicture}
    \def\length{5}
    \def\width{3}
    \def\featureLength{0.75}
    \def\featureWidth{1}
    \draw[thick] (0, 0) rectangle (\length, \width);

    \coordinate (interfaceBottom) at (\length / 2, 0);
    \coordinate (interfaceTop) at (\length / 2, \width);
    \coordinate (featureA) at (\length / 2 - \featureLength, \width / 2 - \featureWidth /2);
    \coordinate (featureB) at (\length / 2, \width / 2 - \featureWidth /2);
    \coordinate (featureC) at (\length / 2, \width / 2 + \featureWidth /2);
    \coordinate (featureD) at (\length / 2 - \featureLength, \width / 2 + \featureWidth /2);
    
    \draw[thick] (interfaceBottom) -- (featureB) -- (featureA) -- (featureD) -- (featureC) -- (interfaceTop);

    \node at (\length / 8, 0.5 * \width) {$\domaincomp{1}$};
    \node at (7 * \length / 8, 0.5 * \width) {$\domaincomp{2}$};

    \node[anchor=west] at ($(interfaceBottom)!0.85!(interfaceTop)$) {$\interface$};

\end{tikzpicture}
    \caption{}
    \label{fig:model problem:exact domain}
    \end{subfigure}\hspace{1cm}%
    \begin{subfigure}[B]{0.42\textwidth}
    \centering
    \begin{tikzpicture}
    \def\length{5}
    \def\width{3}
    \def\featureLength{0.75}
    \def\featureWidth{1}
    \draw[thick] (0, 0) rectangle (\length, \width);

    \coordinate (interfaceBottom) at (\length / 2, 0);
    \coordinate (interfaceTop) at (\length / 2, \width);
    \coordinate (featureA) at (\length / 2 - \featureLength, \width / 2 - \featureWidth /2);
    \coordinate (featureB) at (\length / 2, \width / 2 - \featureWidth /2);
    \coordinate (featureC) at (\length / 2, \width / 2 + \featureWidth /2);
    \coordinate (featureD) at (\length / 2 - \featureLength, \width / 2 + \featureWidth /2);

    \draw[thick] (interfaceBottom) -- (featureB);
    \draw[thick] (interfaceTop) -- (featureC);

    \fill[color=gray!15]  (featureA) rectangle (featureC);
    \node at ($(featureA)!0.5!(featureC)$) {$\patch$};

    \draw[thick, dashed, blue1] (featureB) -- (featureA) -- (featureD) -- (featureC);
    \node[anchor=east,blue1] at ($(featureA)!0.5!(featureD)$) {$\defbd$};

    \draw[thick, red] (featureB) -- (featureC);
    \node[anchor=west,red] at ($(featureB)!0.5!(featureC)$) {$\simpbd$};

    \node at (\length / 8, 0.5 * \width) {$\comp{\domain}{1}$};
    \node at (7 * \length / 8, 0.5 * \width) {$\comp{\defdomain}{2}$};

    \node[anchor=west] at ($(interfaceBottom)!0.85!(interfaceTop)$) {$\definterface$};

\end{tikzpicture}
    \caption{}
    \label{fig:model problem:defeatured domain}
    \end{subfigure}
    \caption{Example of a computational domain $\Omega$: (a) exact geometry with the components $\domaincomp{1}$ and $\domaincomp{2}$ separated by the interface $\Gamma$; (b) coarsened-interface geometry with the feature patch $\patch$ and the components $\comp{\defdomain}{1} = \mathrm{int}(\overline{\domaincomp{1}} \cup \overline{\patch})$ and $\comp{\defdomain}{2} = \domaincomp{2} \setminus \overline{\patch}$ separated by the interface $\definterface$.}
    \label{fig:model problem:domains}
\end{figure}
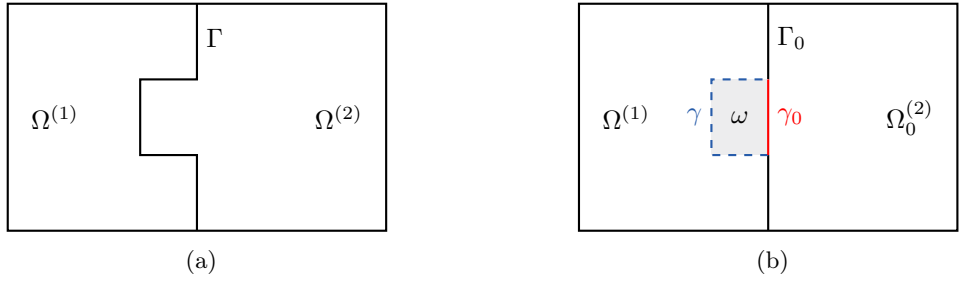

Consider the PDE problem
\begin{align*}
    \begin{cases}
        -\nabla \cdot (\diffcoeff \nabla u) = f & \text{ in } \domain,
        \\
        u = 0 & \text{ on } \partial \domain,
    \end{cases}
\end{align*}
for a piecewise constant diffusion coefficient $\diffcoeff \in L^\infty(\domain)$ such that
\begin{align*}
    \diffcoeff(\bx) \coloneqq \begin{cases}
        \kappa_1 & \text{if }\bx\in \domaincomp{1},
        \\
        \kappa_2 & \text{if }\bx\in \domaincomp{2},
    \end{cases}
\end{align*}
with $\kappa_i > 0, i \in \{1, 2\}$.
The corresponding weak formulation reads: \emph{Find $u \in \honez{\domain}$ such that}
\begin{align}
\label{eq:model problem:weak form}
    a(u, v) = (f, v )_\domain \qquad \forall v \in \honez{\domain},
\end{align}
where the bilinear form $a: \honez{\domain} \times \honez{\domain} \to \R$ is defined by
\begin{align*}
    a(u, v) \coloneqq( \diffcoeff\nabla u \cdot \nabla v)_\domain =( \kappa_1\nabla u,\nabla v)_{\domaincomp{1}} + ( \kappa_2\nabla u, \nabla v)_{\domaincomp{2}}.
\end{align*}
The problem in \cref{eq:model problem:weak form} is well-posed by the Lax-Milgram Theorem \cite{ern_finite_2021} and is referred to as the \emph{exact problem}.
The energy norm associated with the bilinear form used in \cref{eq:model problem:weak form} is defined by
\begin{align}
\label{eq:model problem:energy norm}
    \energynorm{v}{a} \coloneqq a(v, v)^\frac{1}{2},
\end{align}
for any function $v \in \honez{\domain}$.

Let us now consider a simplified decomposition of the domain $\domain$, in which a \emph{coarsened} interface $\Gamma_0$ separates $\Omega$ into two components $\defdomaincomp{1}$ and $\defdomaincomp{2}$, such that $\domain=\interior{\overline{\defdomaincomp{1}}\cup \overline{\defdomaincomp{2}}}$; see \cref{fig:model problem:defeatured domain}. We assume that $\Gamma \cap \Gamma_0 \neq \emptyset$ and we define \begin{align*}
    \defbd \coloneqq \interface \setminus \overline{\definterface} \quad \text{ and } \quad \simpbd \coloneqq \definterface \setminus\overline{\interface}.
\end{align*}
We assume that $\partial \omega=\mathrm{int}(\overline{\gamma}\cup \overline{\gamma_{0}})$ is a closed curve corresponding to the boundary of an open domain $\omega$.
In the following, we will refer to $\omega$ as \emph{interface feature}, or shortly as \emph{feature}. For this reason, we will refer to the coarsened interface $\Gamma_0$ also as the \emph{defeatured interface}.
We now introduce the problem with the coarsened interface, referred to as the \emph{defeatured problem}, by defining a new diffusion coefficient\vspace{0.2cm}
\begin{align*}
    \defdiffcoeff(\bx) \coloneqq \begin{cases}
        \kappa_1 & \text{if }\bx \in \comp{\defdomain}{1},
        \\
        \kappa_2 & \text{if }\bx \in \comp{\defdomain}{2},
    \end{cases}
\end{align*}\vspace{0.2cm}
and the bilinear form
$a_0:H_0^1(\domain) \times H_0^1(\domain) \rightarrow \mathbb{R}$ such that\vspace{0.1cm}
\begin{align*}
    a_0(u, v) \coloneqq (\defdiffcoeff\nabla u,\nabla v )_\domain = (\kappa_1\nabla u, \nabla v)_{\comp{\defdomain}{1}} + (\kappa_2\nabla u, \nabla v)_{\comp{\defdomain}{2}}.
\end{align*}
\vspace{0.2cm}
The weak formulation of the \emph{defeatured problem} reads: \emph{Find $u_0\in \honez{\domain}$ such that}
\begin{align}
\label{eq:defeatured problem:weak form}
    a_0(u_0, v) = (f, v)_\domain \qquad \forall v \in \honez{\domain}.
\end{align}
with the corresponding energy norm defined by
\begin{align}
    \energynorm{v}{a_0} \coloneqq a_0(v, v) ^\frac{1}{2} & \qquad\forall v \in \honez{\domain}.
\end{align}

Let us now consider a non-degenerate simplicial mesh $\T{h}$ covering $\domain$, such that $\overline{\domain}=\bigcup_{K \in \T{h}}K$. Here, we assume that the mesh faces match the defeatured interface $\Gamma_0$ or, equivalently, that $\Gamma_0$ is an approximation of $\Gamma$ that is easy to mesh conformingly. Instead, we make no assumption about how the mesh intersects the feature $\patch$ and its boundary $\partial \patch$.

Let us introduce the space
\begin{equation}Q_h(\domain)=\mathcal{P}_p(\mathcal{T}_h^0):=\left\lbrace q_h\in L^2(\domain):~\tr{q_h}{K}\in \mathcal{P}_p(K),~\forall K \in \mathcal{T}_h^0 \right\rbrace,\label{eq:Qh0}
\end{equation} 
with $\mathcal{P}_p(K)$ denoting the space of polynomials of degree at most $p\geq 1$ on $K \in \mathcal{T}_h^0$, and let
$$V_h^0(\domain)=\lbrace v \in \mathcal{C}^0(\overline{\domain})\cap Q_h(\domain): \tr{v}{\partial \domain} =0\rbrace.$$
For the sake of simplicity, we assume $f \in Q_h(\domain)$.
The discrete version of the defeatured problem can then be written as:
\emph{find $u_{0h} \in V_h^0(\domain)$ such that}
\begin{equation}
a_0( u_{0h}, v_h)=(f,v_h)_\domain \qquad \forall v_h \in V_h^0(\domain)\label{eq:prob_num_def}.
\end{equation}

\section{A defeaturing error estimator based on equilibrated fluxes}\label{sec:flux}
Given the respective solutions $u\in H_0^1(\domain)$, $u_0\in H_0^1(\domain)$, and $\uhz \in V_h^0(\domain)$ to the exact problem \eqref{eq:model problem:weak form}, to the defeatured problem \eqref{eq:defeatured problem:weak form}, and to the discrete defeatured problem \eqref{eq:prob_num_def}, we define the error function
\begin{align*}
    e \coloneqq u-\uhz.
\end{align*}
This error function can be decomposed as $e=e_\mathrm{d}+e_\mathrm{n}$. The first component is defined by
\begin{align*}
e_\mathrm{d}\coloneqq u-u_0
\end{align*}
and it measures the modeling error we introduce when we approximate the exact solution $u$ with the defeatured solution $u_0$. We will refer to it as the \textit{defeaturing component} of the error. The second component is instead
$$e_\mathrm{n}\coloneqq u_0-\uhz,$$
and it corresponds to the numerical source of error, introduced by approximating $u_0$ by its discrete counterpart $\uhz$.
Our aim is to bound the energy norm of the whole error function $e$ with a quantity solely depending on the discrete solution of the defeatured problem $\uhz$. In particular, we aim at building an \emph{a posteriori} error estimator based on an \emph{equilibrated flux} reconstructed from $\uhz$.

Let $\bm{\sigma}\in \Hdiv{\domain}$ be such that 
\begin{equation}
\nabla \cdot \bm{\sigma}=f \quad \text{ in } \domain.
\label{eq:flux_prop}
\end{equation}
By the Prager and Synge theorem (see \cite{braess_finite_2007,prager_approximations_1947}) we know that
$$\energynorm{e_\mathrm{n}}{a_0}\leq||\defdiffcoeff^{-\frac{1}{2}}(\bm{\sigma}+\defdiffcoeff\nabla \uhz)||_\domain,$$
which means that we can provide, through $\bm{\sigma}$, a sharp bound on the numerical component of the error, with a reliability constant equal to 1. Assuming we are able to reconstruct a flux $\flux\in \Hdiv{\domain}$ satisfying \eqref{eq:flux_prop} at a discrete level, we want to define an estimator $\mathcal{E}(\flux)$ such that
\begin{align*}
||e||_a\leq C_a \, \mathcal{E}(\flux),
\end{align*}
for a positive constant $C_a > 0$ independent of the size of the feature $\patch$.

In the remainder of this section, we first provide some details on the construction of $\flux$, and we then go through the derivation of $\estim{}(\flux).$

\subsection{Equilibrated flux reconstructions}
The procedure adopted to reconstruct $\flux$ from $\uhz$ is summarized here. For details, we refer to \cite{braess_finite_2007, buffa_equilibrated_2024, ern_polynomial-degree-robust_2015}.
An \textit{equilibrated flux reconstruction} is any function $\flux$ reconstructed from $\uhz$ such that \cite{ern_polynomial-degree-robust_2015}
\begin{align*}
\flux \in \Hdiv{\domain} \quad \text{ and } \quad (\nabla \cdot \flux, 1)_K=(f,1)_K\quad \forall K \in \T{h}.
\end{align*}
We here reconstruct $\flux$ in the Raviart-Thomas space of order $p$, i.e., in 
$$\bm{M}_h(\domain)\coloneqq \left\lbrace  \bm{v}_h\in \Hdiv{\domain}:~\tr{\bm{v}_h}{K} \in [\mathcal{P}_p(K)]^d+\bm{x}\mathcal{P}_p(K),~\forall K \in \T{h}\right\rbrace.$$ 
We recall that we assume $\T{h}$ to be conforming to the interface $\Gamma_0$, meaning that the diffusion coefficient $\defdiffcoeff$ always assumes a unique value on each $K\in \T{h}$.
Following the approach proposed in \cite{braess_equilibrated_2007,ern_polynomial-degree-robust_2015}, we define $\flux$ as the sum of local contributions reconstructed on patches $\wa$ of elements sharing a node $\bm{a} \in \Nh$, with
$\mathscr{N}_h$ denoting the set of vertices in $\T{h}$. This approach has the fundamental advantage of being easily parallelizable, as it is based on solving independent local sub-problems defined on patches of a few elements.

Let us denote by $\psi_{\bm{a}}$ the \textit{hat} function in $Q_h(\domain)\cap H^1(\Omega_0)$ taking value 1 in vertex $\bm{a}$ and 0 on all the other vertices. The boundary of a patch $\wa$ is denoted by $\partial \wa$, whereas we define $\partial \pi_{\bm{a}}^0\subseteq \partial \wa$ as
$$\partial \pi_{\bm{a}}^0=\lbrace \bm{x} \in \partial \wa~:\psi_{\bm{a}}(\bm{x})=0\rbrace.$$ Let us denote by $\Nh^{\mathrm{int}}$ the mesh nodes lying in the interior of $\domain$, and by $\Nh^\partial$ the ones lying on the boundary $\partial \domain$. We then introduce the spaces
$$\Mha=\left\lbrace \bm{v}_h \in \bm{M}_h(\wa):~\bm{v}_h\cdot \bm{n}_{\wa}=0 \text{ on }\partial \pi_{\bm{a}}^0\right\rbrace,$$
and
\begin{align*}
\Qha&\coloneqq \begin{cases}\lbrace q_h \in Q_h(\wa):(q_h,1)_{\wa}=0 \rbrace & \text{ if } \bm{a} \in \Nh^{\mathrm{int}}\\
Q_h(\wa) &\text{ if } \bm{a} \in \Nh^{\partial},
\end{cases}
\end{align*}
where  $\bm{M}_h(\wa)$ and $Q_h(\wa)$ are, respectively, the restrictions of $\bm{M}_h$ and $Q_h$ to the patch $\wa$.
 We then look for local equilibrated flux reconstructions of the form
\begin{equation}
\fluxa=\arg\min_{\bm{v}_h \in \Mha}||\kappa_0^{-\frac{1}{2}}\bm{v}_h+\psi_{\bm{a}}\defdiffcoeff^\frac{1}{2}\nabla u_0^h||_{\wa}\quad \text{ subject to }\quad  \nabla \cdot \bm{v}_h =\psia f-\nabla\psia\cdot\defdiffcoeff \nabla u_0^h \label{eq:local_argmin},
\end{equation} and then we set $\flux =\sum_{\bm{a}\in \Nh}\fluxa.$
The optimization problem \eqref{eq:local_argmin} is equivalent to looking for $\fluxa \in \Mha$ and $\lam \in \Qha$ such that,
\begin{align}
(\kappa_0^{-1}\fluxa,\bm{v}_h)_\wa-(\lam,\nabla \cdot \bm{v}_h)_\wa=-(\psia\nabla u_0^h,\bm{v}_h)_\wa \quad \forall \bm{v}_h \in \Mha \label{eq:mixed1}\\[0.5em]
(\nabla \cdot \fluxa,q_h)_\wa=(\psia f,q_h)_\wa-(\nabla\psia\cdot \defdiffcoeff\nabla u_0^h,q_h)_\wa \quad \forall q_h \in \Qha\label{eq:mixed2},
\end{align}
As shown in Lemma 3.5 \cite{ern_polynomial-degree-robust_2015}, the resulting flux reconstruction $\flux \in \Hdiv{\domain}$ is such that 
\begin{equation}
    (\nabla \cdot \flux,v_h)_K=(f,v_h)_K \quad \forall v_h \in \mathcal{P}_p(K). \label{eq:flux_orthogonality}
\end{equation}
\begin{rem}
    Requiring the mesh to conform to the coarsened interface enables a straightforward computation of both the discrete defeatured solution $\uhz$ and the corresponding equilibrated flux reconstruction $\flux$. It is, however, worth noting
    that an unfitted approach could be adopted to allow the coarsened interface to cut the mesh arbitrarily. In this case, the variational formulation of the defeatured discrete problem and the flux reconstruction procedure must be adapted in order to weakly impose the interface conditions, leading to some extra terms in the estimator (we refer to \cite{capatina_elliptic_2025, capatina_elliptic_2026} for details on the reconstruction of equilibrated fluxes involving unfitted interfaces).
    Hence, we select, among many possible coarsened interfaces, one that readily allows for a high-quality conforming mesh, yielding the greatest savings in computational effort.
\end{rem}

\subsection{The error estimator}
We now introduce the combined error estimator based on the equilibrated flux reconstruction from the previous section. To that end, let $$\bm{s}_0\coloneqq \kappa_0^{-\frac{1}{2}}\flux+\kappa_0^\frac{1}{2}\nabla u_0^h,$$
and let us define, for any $K \in \T{h}$, the quantity
$$\estim{0}^K\coloneqq ||\bm{s}_0||_K.$$
 Additionally, let us then define the positive constant $\rk\coloneqq \ktwokone$ and the subset of mesh elements having non-empty intersection with the feature $\patch$,
\begin{equation}
\mathcal{W}_h\coloneqq \lbrace K\in \T{h}:~|K\cap \patch|\neq 0\rbrace. \label{eq:Wh}
\end{equation}
In the following, we define and prove the reliability of our \textit{a posteriori} error bound.
\begin{proposition}
\label{prop:estimator reliability}
    Let $u$ be the solution to \eqref{eq:model problem:weak form} and $\uhz$ the solution to \eqref{eq:prob_num_def}. Then
    \begin{equation}
        \energynorm{e}{a}\leq \estim{\mathrm{d}}+\estim{\mathrm{n}}\label{eq:estimator}
    \end{equation}
    with 
    $$\estim{\mathrm{d}}=C(\kappa)\big(\sum_{K \in \mathcal{W}_h}||\flux||_{K \cap \patch}^2\big)^\frac{1}{2},$$
    $$
    \estim{\mathrm{n}}=\big(\sum_{K \in \mathcal{T}_h}||\bm{s}_0||_K^2\big)^\frac{1}{2}+\widetilde{C}(\kappa)\big(\sum_{K \in \mathcal{W}_h}||\bm{s}_0||_{K \cap \patch}^2\big)^\frac{1}{2},$$
    
    where $C(\kappa)=|\rho-1|\kappa_2^{-\frac{1}{2}}$  and  $\widetilde{C}(\kappa)=|\rho^\frac{1}{2}-1|$.
\end{proposition}
    
    \begin{proof}
Let $v \in \honez{\domain}$. Adding and subtracting $(\kkz\flux,\nabla v)_\domain$ we obtain
\begin{align}
a(e,v)&=(\diffcoeff\nabla(u-\uhz),\nabla v)_\Omega\nonumber\\&=(\diffcoeff\nabla u+\kkz\flux,\nabla v)_\domain-(\kkz\flux+\diffcoeff\nabla \uhz,\nabla v)_\domain\nonumber\\&=\mathrm{I}_a+\mathrm{I}_b\label{eq:a_e_v}
\end{align}
with
\begin{align*}
&\mathrm{I}_a=(\diffcoeff\nabla u+\kkz\flux,\nabla v)_\domain\\
&\mathrm{I}_b=-(\kkz\flux+\diffcoeff \nabla \uhz,\nabla v)_\domain.
\end{align*}
Starting from $\mathrm{I}_a$, let us first of all observe that
$$\kkz(\bm{x})=\begin{cases}
    1 &\text{ if } \bm{x} \in\domain \setminus \patch\\
    \rk &\text{ if }\bm{x} \in\patch,
\end{cases}\vspace{-0.5cm}$$ 
so that \begin{align*}
    \mathrm{I}_a&=(\kappa\nabla u+\flux, \nabla v)_{\domain\setminus \patch}+(\kappa\nabla u+\rho\flux, \nabla v)_{\patch}.
\end{align*}
If we denote by $\nn{\patch}$ and $\nn{\Omega\setminus\patch}$ the outward pointing unit normals respectively to the feature $\patch$ and to $\Omega \setminus\patch$, we observe that $\nn{\patch}=-\nn{\Omega\setminus\patch}$ on $\partial \patch$.
Then, integrating by parts both terms in $\mathrm{I}_a$, and using Equation \eqref{eq:model problem:weak form}, we obtain
\begin{align}
\mathrm{I}_a&=(f-\nabla \cdot \flux,v)_{\domain\setminus \patch}+(f-\rk\nabla \cdot \flux,v)_{\patch}+\langle (\rk-1)\flux \cdot \nn{\omega},v\rangle_{\partial \patch}\nonumber\\[0.3em]
    &=(f-\nabla \cdot \flux,v)_{\domain}-(f-\nabla \cdot \flux,v)_{\patch}+(f-\rk\nabla \cdot \flux,v)_{\patch}\nonumber\\[-0.3em]&\quad+\langle (\rk-1)\flux \cdot \nn{\omega},v\rangle_{\partial \patch}\nonumber\\[0.3em]&=(f-\nabla \cdot \flux,v)_{\domain}-((\rk-1)\nabla \cdot \flux,v)_{\patch}+\langle (\rk-1)\flux \cdot \nn{\omega},v\rangle_{\partial \patch}\nonumber\\[0.3em]
    &=(f-\nabla \cdot \flux,v)_{\domain}+((\rk-1)\flux,\nabla v)_{\patch},
    \label{step1_Ia}
    \end{align}
    where, in the last step, we have again used integration by parts.
Finally, by the Cauchy--Schwarz inequality, we obtain
\begin{align}
    \mathrm{I}_a&\leq\big(\hspace{-0.1cm}\sum_{K \in \T{h}}||f-\nabla \cdot \flux||_K^2\big)^\frac{1}{2} ||v||_\domain + \big(\hspace{-0.1cm}\sum_{K \in \mathcal{W}_h}||(\rk-1)\kappa_2^{-\frac{1}{2}}\flux||_{K \cap \patch}^2\big)^\frac{1}{2}(\kappa_2^\frac{1}{2}||\nabla v||_\patch)\nonumber \\&\leq
   |\rk-1|\kappa_2^{-\frac{1}{2}}\big(\sum_{K \in \mathcal{W}_h}||\flux||_{K \cap \patch}^2\big)^\frac{1}{2}\energynorm{v}{a},
\end{align}
where the first term vanishes according to the orthogonality relation \eqref{eq:flux_orthogonality}, as we have assumed $f \in Q_h$ and as $\nabla \cdot \bm{M}_h = Q_h$. 
Moving back to the term $\mathrm{I}_b$ in \eqref{eq:a_e_v}, we have
\begin{align}
\mathrm{I}_b&=-(\kkz\flux+\kappa\nabla u_0^h,\nabla v)_{\domain}=-((\kkz)^\frac{1}{2}\bm{s}_0,\kappa^\frac{1}{2}\nabla v)_{\domain}\nonumber\\&=-(\bm{s}_0,\kappa^\frac{1}{2}\nabla v)_{\domain\setminus \patch}-(\rk^\frac{1}{2}\bm{s}_0,\kappa^\frac{1}{2}\nabla v)_{\patch}\nonumber\\
&=-(\bm{s}_0,\kappa^\frac{1}{2}\nabla v)_{\domain}-((\rk^\frac{1}{2}-1)\bm{s}_0,\kappa^\frac{1}{2}\nabla v)_{\patch}\nonumber
\end{align} 
Then, by the Cauchy--Schwarz inequality, we obtain
\begin{align*}
    \mathrm{I}_b  &\leq \big(\sum_{K \in \T{h}}||\bm{s}_0||_K^2\big)^\frac{1}{2}\energynorm{v}{a}+|\rk^\frac{1}{2}-1|\big(\sum_{K \in \mathcal{W}_{h}}||\bm{s}_0||_{K\cap \omega}^2\big)^\frac{1}{2}\energynorm{v}{a}.
\end{align*}
Gathering the above estimates proves the assertion.
\end{proof}

\begin{rem}
In principle, integrating the quantity $\bm{s}_0$ over the intersections $K \cap \omega$ for the elements $K \in \mathcal{W}_h$ would require tailored quadrature rules, since the coarse mesh is not conforming to the feature $\omega$. This would, however, increase the computational effort required to assemble the estimator.
The crudest surrogate is to integrate over the whole element $K$ rather than over the intersection $K \cap \omega$, but this introduces significant noise into the defeaturing component of the estimator whenever the feature is small with respect to the mesh size. We instead follow an intermediate approach:
given a function $g_K : K \to \R$ for an element $K \in \mathcal{W}_h$, we approximate
\begin{align*}
    \int_{K \cap \omega} g_K \dd x \approx \theta_K \int_K g_K \dd x,
    \qquad \text{with} \qquad \theta_K \coloneqq \frac{|K \cap \omega|}{|K|}.
\end{align*}
If we assume that the feature is a polygon, or consider a polygonal approximation of it,  the volume $|K \cap \omega|$ can be evaluated easily. Let us remark that this simplified integration strategy does not affect either the discrete solution of the defeatured problem or the equilibrated flux reconstruction, but only the estimator, for which we aim for a cheap assembly.
\end{rem}

\begin{rem}
\label{rem:estimator:inclusions}
    Let us observe that \cref{prop:estimator reliability} holds also when $\Gamma_0=\emptyset$, i.e., when $\defdomaincomp{1}\equiv \Omega$ and $\defdomaincomp{2}=\emptyset$, as shown in~\cref{fig:numerical experiments:inclusion}. The proof of \cref{prop:estimator reliability} can be readily adapted by simply assuming $\patch=\Omega_2$. In this case, $u_{0h}$ and the corresponding equilibrated flux reconstruction $\flux$ can be built on a mesh that is completely unaware of the feature. The spirit is hence similar to that of \cite{buffa_equilibrated_2024}, with the difference that the features are assumed to be filled with a different material rather than being holes.
\end{rem}

\subsection{Generalization to the multiple feature case}
The estimate \eqref{eq:estimator} readily generalizes to the multiple feature case. Let us assume that $\gamma$ and $\gamma_0$ can both be written as the union of a finite number of disjoint components,
$$\gamma=\bigcup\nolimits_{i \in \I}\gamma_i \qquad \text{and} \qquad \gamma_0=\bigcup\nolimits_{i \in \I}\gamma_{0i} $$ such that
\begin{enumerate}
    \item[(i)] $\gamma_i \cap \gamma_j =\emptyset$ and $\gamma_{0i} \cap \gamma_{0j}=\emptyset$ for any $i,j \in \I$, $i \neq j$;
    \item[(ii)] $\forall i \in \I, \partial \patch_i=\overline{\gamma_i}\cup \overline{ \gamma_{0i}}$ is a closed curve corresponding to the boundary of an open Lipschitz domain $\patch_i$. 
\end{enumerate}
Let us define the index subsets
$$\I^{(1)}\coloneqq \lbrace i \in \I:~\patch_i \cap \defdomaincomp{1}\neq \emptyset\rbrace \quad \text{and} \quad \I^{(2)}\coloneqq \lbrace i \in \I:~\patch_i \cap \defdomaincomp{2}\neq \emptyset\rbrace, $$ and let $$\patch^{(1)}\coloneqq \bigcup\nolimits_{i \in \I^{(1)}}\omega_i \quad \text{and} \quad \patch^{(2)}\coloneqq \bigcup\nolimits_{i \in \I^{(2)}}\omega_i .$$ Next we define
\begin{equation}
\mathcal{W}_h^{(1)}\coloneqq \lbrace K\in \T{h}:~|K\cap \patch^{(1)}|\neq 0\rbrace \quad \text{and} \quad \mathcal{W}_h^{(2)}\coloneqq \lbrace K\in \T{h}:~|K\cap \patch^{(2)}|\neq 0\rbrace.
\end{equation}
Let us observe that $\mathcal{W}_h^{(1)}$ and $\mathcal{W}_h^{(2)}$ are disjoint sets, as we are assuming the mesh $\T{h}$ to be conforming to $\Gamma_0$. The estimate in \eqref{eq:estimator} is hence still valid, provided that $\estim{\mathrm{d}}$ and $\estim{\mathrm{n}}$ are redefined by
$$\estim{\mathrm{d}}\coloneqq C(\kappa)\Big(\sum_{K \in \mathcal{W}_h^{(1)}}||\flux||_K^2\Big)^\frac{1}{2}+C(\kappa)\rho^{-\frac{1}{2}}\Big(\sum_{K \in \mathcal{W}_h^{(2)}}||\flux||_K^2\Big)^\frac{1}{2} $$ 
and 
$$\estim{\mathrm{n}}\coloneqq \big(\sum_{K \in \mathcal{T}_h}||\bm{s}_0||_K^2\big)^\frac{1}{2}+\widetilde{C}(\kappa)\big(\sum_{K \in \mathcal{W}_h^{(1)}}||\bm{s}_0||_K^2\big)^\frac{1}{2}+\widetilde{C}(\kappa)\rho^{-\frac{1}{2}}\big(\sum_{K \in \mathcal{W}_h^{(2)}}||\bm{s}_0||_K^2\big)^\frac{1}{2},$$
where $\rho=\tfrac{\kappa_2}{\kappa_1}$, and the scaling by $\rho^{-\frac{1}{2}}$ for the terms involving $K \in \mathcal{W}_h^{(2)}$ derives from the fact that $\kkz={\rho}^{-1}$ in $\patch^{(2)}$.
\section{Numerical experiments}
\label{sec:numerical experiments}
We present numerical tests to validate the proposed error estimator. All results are obtained with polynomial order $p=1$; i.e., we use linear Lagrange finite elements to compute the primal solution $\uhz$ and reconstruct the corresponding equilibrated flux in a Raviart-Thomas space of order 1.  

The numerical defeatured solution and the equilibrated flux reconstruction are always computed on meshes that conform to the coarsened interface $\Gamma_0$ but intersect $\patch_i$ arbitrarily.
Errors are computed with respect to a reference solution, which we obtain by solving the exact problem numerically on a highly refined mesh conforming to the original interface $\Gamma$.
The experiments are implemented using the \texttt{FEniCSx} library~\cite{baratta_dolfinx_2023} with meshes generated by \texttt{gmsh}~\cite{geuzaine_gmsh_2009}.

The geometries considered in the first three experiments are reported in \cref{fig:numerical experiments:geometries}. In test T1 (\cref{fig:numerical experiments:geometries:sqinsq}) we analyze the behavior of the error estimator under mesh refinement and feature size variation; in test T2 (\cref{fig:numerical experiments:geometries:leftright}), instead, we fix the feature and the mesh size and we vary the contrast between the diffusion coefficients; in T3 (\cref{fig:numerical experiments:geometries:multi}) we use the proposed estimator in a multifeature scenario as a tool to identify the most relevant features at a given mesh size. In the fourth test, labeled T4, we finally use the proposed estimator in a different setting, using it as a tool to assess the impact of a material inclusion embedded in a domain with a different diffusion coefficient.
\begin{figure}[t]
    \centering
    \begin{subfigure}{0.25\textwidth}
    \centering
    \begin{tikzpicture}[scale=0.5]
    \def\outer{5}
    \def\inner{2}
    \def\featureWidth{.6}
    \filldraw[thick, fill=gray!15] (-\outer / 2, -\outer / 2) rectangle (\outer / 2, \outer / 2);
    \filldraw[thick, fill=white] (-\inner / 2, -\inner / 2) rectangle (\inner / 2, \inner / 2);

    \node[anchor=south west] at (-\outer / 2, -\outer / 2) {$\domaincomp{2}$};
    \node[anchor=south west] at (-\inner / 2, -\inner / 2) {$\defdomaincomp{1}$};

    \coordinate (featureA) at (\featureWidth / 2, \inner / 2);
    \coordinate (featureB) at (\featureWidth / 2, \inner / 2 + \featureWidth);
    \coordinate (featureC) at (-\featureWidth / 2, \inner / 2 + \featureWidth);
    \coordinate (featureD) at (-\featureWidth / 2, \inner / 2);
    
    \fill[color=white] (-\featureWidth / 2, \inner / 2) rectangle (\featureWidth / 2, \inner / 2 + \featureWidth);
    \node[font=\scriptsize] at ($(featureA)!0.5!(featureC)$) {$\patch$};

    \draw[thick, red1] (featureD) -- (featureA);
    \node[red1, anchor=north, font=\scriptsize] at ($(featureA)!0.5!(featureD)$) {$\simpbd$};
    
    \draw[thick, blue1] (featureA) -- (featureB) -- (featureC) -- (featureD);
    \node[blue1, anchor=south, font=\scriptsize] at ($(featureB)!0.5!(featureC)$) {$\defbd$};

    \node[anchor=west] at (\inner / 2, 0) {$\definterface$};
    
\end{tikzpicture}
    \caption{Test T1}
    \label{fig:numerical experiments:geometries:sqinsq}
    \end{subfigure}\medskip\\
    \begin{subfigure}{0.5\textwidth}
    \centering
    \begin{tikzpicture}[scale=0.8]
    \def\length{6}
    \def\width{3}
    \def\featureLength{0.75}
    \def\featureWidth{0.75}
    \fill[color=gray!15] (\length / 2, 0) rectangle  (\length, \width);
    \draw[thick] (0, 0) rectangle (\length, \width);

    \coordinate (interfaceBottom) at (\length / 2, 0);
    \coordinate (interfaceTop) at (\length / 2, \width);
    \coordinate (featureA) at (\length / 2 - \featureLength, \width / 2 - \featureWidth /2);
    \coordinate (featureB) at (\length / 2, \width / 2 - \featureWidth /2);
    \coordinate (featureC) at (\length / 2, \width / 2 + \featureWidth /2);
    \coordinate (featureD) at (\length / 2 - \featureLength, \width / 2 + \featureWidth /2);

    \draw[thick] (interfaceBottom) -- (featureB);
    \draw[thick] (interfaceTop) -- (featureC);

    \fill[color=gray!15]  (featureA) rectangle (featureC);
    \node at ($(featureA)!0.5!(featureC)$) {$\patch$};

    \draw[thick, blue1] (featureB) -- (featureA) -- (featureD) -- (featureC);
    \node[anchor=east,blue1] at ($(featureA)!0.5!(featureD)$) {$\defbd$};

    \draw[thick, red] (featureB) -- (featureC);
    \node[anchor=west,red] at ($(featureB)!0.5!(featureC)$) {$\simpbd$};

    \node at (\length / 8, 0.5 * \width) {$\comp{\domain}{1}$};
    \node at (7 * \length / 8, 0.5 * \width) {$\comp{\defdomain}{2}$};

    \node[anchor=west] at ($(interfaceBottom)!0.85!(interfaceTop)$) {$\definterface$};

\end{tikzpicture}
    \caption{Test T2}
    \label{fig:numerical experiments:geometries:leftright}
    \end{subfigure}\hfill
    \begin{subfigure}{0.5\textwidth}
    \centering
    \begin{tikzpicture}[scale=0.65]
    \def\totalFeatures{6}
    \def\length{6}
    \def\width{3.7}
    \def\centerX{\length/2}
    \def\leaderLength{0.7}

    \fill[gray!15] (\centerX, 0) rectangle (\length, \width);

    \draw[thick] (0, 0) rectangle (\length, \width);

    \draw[thick, dotted] (\centerX, 0) -- (\centerX, \width);

    \coordinate (lastPoint) at (\centerX, \width);

    \foreach \i/\s in {
        1/0.35, 
        2/0.3, 
        3/0.2, 
        4/0.2, 
        5/0.3, 
        6/0.35
    } {
        \pgfmathsetmacro{\yPos}{\width - (\i * \width / (\totalFeatures + 1))}

        \pgfmathsetmacro{\side}{isodd(\i) ? 1 : -1}
        \pgfmathsetmacro{\col}{isodd(\i) ? "blue1" : "red1"}
        \pgfmathsetmacro{\fillCol}{isodd(\i) ? "white" : "gray!15"}

        \coordinate (entry) at (\centerX, \yPos + \s/2);
        \coordinate (exit)  at (\centerX, \yPos - \s/2);
        \coordinate (outerCornerTop) at (\centerX + \side*\s, \yPos + \s/2);
        \coordinate (outerCornerBot) at (\centerX + \side*\s, \yPos - \s/2);

        \draw[thick] (lastPoint) -- (entry);

        \fill[\fillCol] (entry) -- (outerCornerTop) -- (outerCornerBot) -- (exit) -- cycle;
        \draw[thick, \col] (entry) -- (outerCornerTop) -- (outerCornerBot) -- (exit);
        
        \ifodd\i
            \draw[thin, \col] (\centerX + \s/2, \yPos) -- (\centerX - \leaderLength, \yPos) node[left, font=\normalsize] {$\patch_{\i}$};
        \else
            \draw[thin, \col] (\centerX - \s/2, \yPos) -- (\centerX + \leaderLength, \yPos) node[right, font=\normalsize] {$\patch_{\i}$};
        \fi

        \coordinate (lastPoint) at (exit);
    }

    \draw[thick] (lastPoint) -- (\centerX, 0);

    \node[anchor=west] at (0.2, \width / 2) {$\domaincomp{1}$};
    \node[anchor=east] at (\length - 0.2, \width / 2) {$\domaincomp{2}$};

\end{tikzpicture}
    \caption{Test T3}
    \label{fig:numerical experiments:geometries:multi}
    \end{subfigure}
    \caption{Geometries considered in the numerical experiments. The shaded and white areas correspond to different diffusion coefficients in the exact problem \eqref{eq:model problem:weak form}}
    \label{fig:numerical experiments:geometries}
\end{figure}
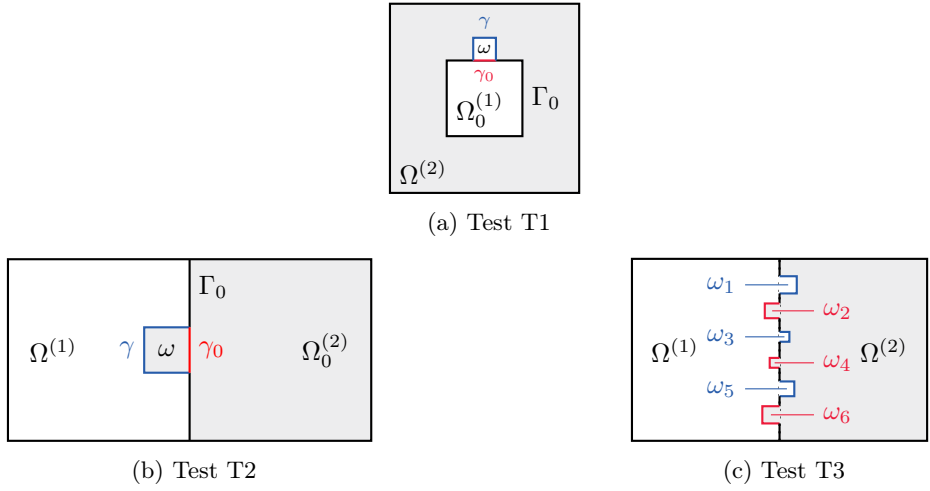

\subsection{T1: Interface coarsening, single feature}
In this first numerical experiment we aim at analyzing the behavior of the proposed error estimator considering the geometry reported in \cref{fig:numerical experiments:geometries:sqinsq}, where $\defdomaincomp{1}$ is a square of edge $l=1$, the feature $\overline{\patch}_{\epsilon}$ is a square of edge $\epsilon<l$ and $\Omega=\mathrm{int}(\overline{\domaincomp{2}}\cup \overline{\defdomaincomp{1}}\cup \overline{\patch_\epsilon})$ is a square of edge $L=2.5$. 
The exact problem is defined as in \eqref{eq:model problem:weak form}, with $\kappa_1=1$, $\kappa_2 \in \lbrace{2,10\rbrace}$ and $f \equiv 1$. 

\Cref{fig:numerical experiments:mesh_kp2} reports the total estimator $\estim{}$, its components 
$\estim{\mathrm{d}}$ and $\estim{\mathrm{n}}$ and the energy norm of the error $||e||_a$ for $\kappa_2=2$, for decreasing mesh sizes $h$ and three different feature sizes. 
\begin{figure}[h]
	\centering
	\hspace*{0.5cm}
	
	\begin{subfigure}[T]{1\textwidth}
		\centering
		\includegraphics[width=0.8\linewidth]{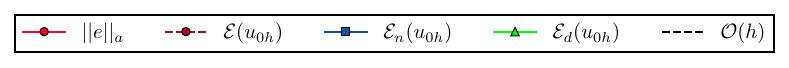}
	\end{subfigure}
	\begin{subfigure}[T]{0.31\textwidth}
		
		\includegraphics[width=1\linewidth]{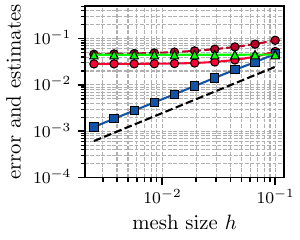}
		\caption{$\epsilon = 0.25$}
		\label{fig:numerical experiments:mesh:big_kp2}
	\end{subfigure}
	\hfill
	\begin{subfigure}[T]{0.31\textwidth}
		\includegraphics[width=1\linewidth]{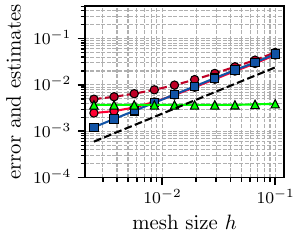}
		\caption{$\epsilon = 0.025$}
		\label{fig:numerical experiments:mesh:intermediate_kp2}
	\end{subfigure}
	\hfill
	\begin{subfigure}[T]{0.31\textwidth}
		\includegraphics[width=1\linewidth]{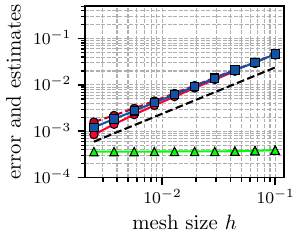}
		\caption{$\epsilon = 0.0025$}
		\label{fig:numerical experiments:mesh:small_kp2}
	\end{subfigure}
	\caption{T1: Energy norm of the error, total estimator and estimator components under mesh refinement and for three fixed feature sizes; $\kappa_2=2$.}
	\label{fig:numerical experiments:mesh_kp2}
\end{figure}
\begin{figure}[H]
	\centering
	\vspace*{-0.1cm}
	\begin{subfigure}[T]{1\textwidth}
		\centering
		\includegraphics[width=0.8\linewidth]{figures/numerics/mesh/legend_estimate.pdf}
	\end{subfigure}
	\begin{subfigure}[T]{0.31\textwidth}
		\includegraphics[width=1\linewidth]{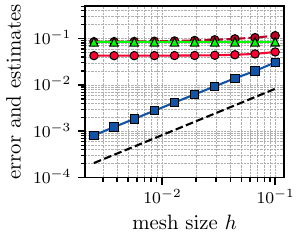}
		\caption{$\epsilon = 0.25$}
		\label{fig:numerical experiments:mesh:big}
	\end{subfigure}
	\hfill
	\begin{subfigure}[T]{0.31\textwidth}
		\includegraphics[width=1\linewidth]{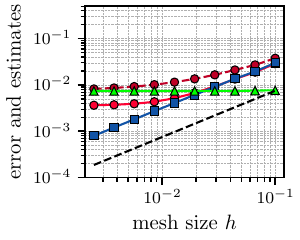}
		\caption{$\epsilon = 0.025$}
		\label{fig:numerical experiments:mesh:intermediate}
	\end{subfigure}
	\hfill
	\begin{subfigure}[T]{0.31\textwidth}
		\includegraphics[width=1\linewidth]{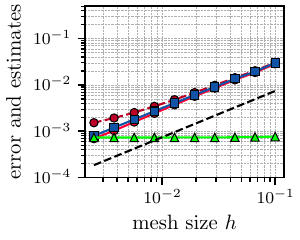}
		\caption{$\epsilon = 0.0025$}
		\label{fig:numerical experiments:mesh:small}
	\end{subfigure}
	\caption{T1: Energy norm of the error, total estimator and estimator components under mesh refinement and for three fixed feature sizes; $\kappa_2=10$.}
	\label{fig:numerical experiments:mesh_kp10}
\end{figure}
\Cref{fig:numerical experiments:mesh:big_kp2} refers to the biggest
considered feature, namely $\epsilon=0.25$. In this case, we observe that, although increasingly finer meshes are considered, the energy norm of the error rapidly reaches a plateau. This indicates that the accuracy cannot be further improved by mesh refinement alone if the feature is neglected. The proposed estimator correctly captures this behavior: although its numerical component $\estim{\mathrm{n}}$ decreases at the expected rate $\mathcal{O}(h)$, the defeaturing component $\estim{\mathrm{d}}$ remains approximately constant, driving the behavior of the total estimator $\estim{}$. \Cref{fig:numerical experiments:mesh:intermediate_kp2} refers instead to an intermediate feature size $\epsilon=0.025$. In this case, we observe that the first steps of mesh refinement reduce the overall error; however, for finer meshes, the defeaturing contribution prevails over the numerical one, and both the error and the estimator approach a plateau. Finally, in \cref{fig:numerical experiments:mesh:small_kp2}, we consider the case in which $\epsilon=0.0025$. Here, we observe that the error and the estimator are primarily driven by their numerical components, with $\estim{\rm{d}}<10^{-3}$. The slight discrepancy between the error and the estimator for the finest meshes in~\cref{fig:numerical experiments:mesh:small_kp2} arises because the mesh resolution approaches that of the reference solution used to compute the exact error. 

Similar results are observed for $\kappa_2=10$ in \cref{fig:numerical experiments:mesh_kp10}, where a higher value of $\estim{\mathrm{d}}$ can be observed for all feature sizes, showcasing how a larger contrast between the diffusion coefficient corresponds to a stronger impact of the feature. A deeper analysis of the effect of the contrast between the coefficients is provided in numerical test T2.

The behavior of the estimator and the energy norm of the error as the feature size decreases is shown in \cref{fig:numerical experiments:size}, for three different mesh sizes and $\kappa_2=2$. As the size of the elements is fixed, the numerical component of the estimator is constant, whereas the defeaturing component decreases as $\mathcal{O}(\varepsilon)$. The numerical component of the estimator almost fully dominates for the coarsest considered meshes ($h=0.075,0.025$), whereas the crossover happens at $\varepsilon \approx 0.02$ in the finest mesh case ($h=0.0075$). 

\begin{figure}
    \centering
    \begin{subfigure}[T]{1\textwidth}
\centering
\includegraphics[width=0.8\linewidth]{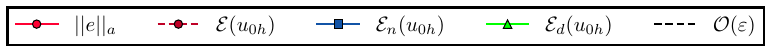}
\end{subfigure}
    \begin{subfigure}[T]{0.31\textwidth}
    \includegraphics[width=1\linewidth]{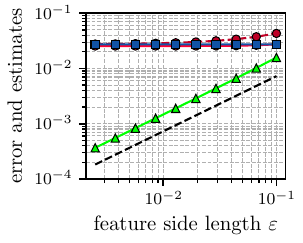}
    \caption{$h = 0.075$}
    \label{fig:numerical experiments:size:coarse}
    \end{subfigure}
    \hfill
    \begin{subfigure}[T]{0.31\textwidth}
    \includegraphics[width=1\linewidth]{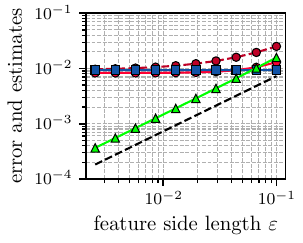}
    \caption{$h = 0.025$}
    \label{fig:numerical experiments:size:intermediate}
    \end{subfigure}
    \hfill
    \begin{subfigure}[T]{0.31\textwidth}
    \includegraphics[width=1\linewidth]{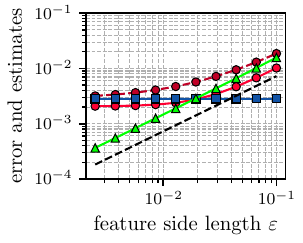}
    \caption{$h = 0.0075$}
    \label{fig:numerical experiments:size:fine}
    \end{subfigure}
    \caption{T1: Results for varying feature sizes in the geometry in \cref{fig:numerical experiments:geometries:sqinsq} for three different mesh sizes.}
    \label{fig:numerical experiments:size}
\end{figure}

\subsection{T2: Interface coarsening, diffusion contrast}
In this second numerical test, we investigate the behavior of the estimator under the variation of the contrast between the diffusion coefficients $\kappa_1$ and $\kappa_2$, considering the geometry reported in \cref{fig:numerical experiments:geometries:leftright}. In particular we choose $\Omega=\rm{int}(\overline{\domaincomp{1}}\cup \overline{\domaincomp{2}})=(-1, 1) \times (-1/2, 1/2)$, whereas the feature $\patch$ is a square of edge $\varepsilon=0.083$ centered at $(-\varepsilon / 2, 0)$. In this geometry, we consider the exact problem \eqref{eq:model problem:weak form} with $f\equiv 1$. The mesh is fixed, with $h=0.005$.

The behavior of the energy norm of the error, of the estimator, and of its components under the variation of the contrast between $\kappa_1$ and $\kappa_2$ is shown in \cref{fig:numerical experiments:coeff}. In particular, in \cref{fig:numerical experiments:coeff:k1}, $\kappa_2$ is kept fixed to 1, whereas we vary $\kappa_1 $ in the interval $[5 \cdot 10^{-1}, 10^4]$. In this case, we can observe how the error plateaus and how the estimator captures this behavior. Let us also remark that, in the limit case of $\kappa_1 \rightarrow\infty$ this scenario corresponds to a classical defeaturing setting (see e.g. \cite{buffa_equilibrated_2024}), in which $\omega$ can be seen as a negative feature at the boundary of $\domaincomp{1}$.

\Cref{fig:numerical experiments:coeff:k2} refers instead to the case in which $\kappa_1$ is kept fixed to 1 whereas $\kappa_2$ varies in the interval $[5 \cdot 10^{-1}, 10^4]$. In this case, it is noteworthy that the error increases with $\kappa_2$ and that the proposed estimator closely follows this trend while remaining reliable across all contrasts.

From the numerical perspective, this behavior can be traced to the two ingredients involved in the defeaturing component, namely $C(\kappa) = |\rho - 1|\,\kappa_2^{-1/2}$ and $\|\flux\|_\patch$. In the experiments where $\kappa_2 = 1$ is held fixed and $\kappa_1 \to \infty$, the factor $C(\kappa)$ tends to one; when instead $\kappa_1 = 1$ is fixed, it grows with $\kappa_2$. As \cref{fig:numerical experiments:coeff_components} shows, the flux norm $\|\flux\|_\patch$ is instead essentially constant across contrasts, and dominates the corresponding numerical quantity $\|\bm{s}_0\|_\patch$. The contrast dependence of $\mathcal{E}_{\rm{d}}$ is, therefore, carried entirely by the rescaling factor $C(\kappa)$.

\begin{figure}[t]
    \centering
    \begin{subfigure}[T]{0.9\textwidth}
    \centering
    \includegraphics[width=1\linewidth]{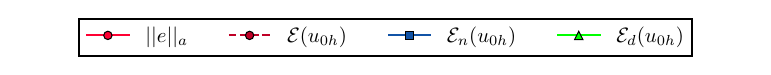}
    \end{subfigure}\\
    \begin{subfigure}[T]{0.45\textwidth}
    \centering
     \includegraphics[width=0.8\linewidth]{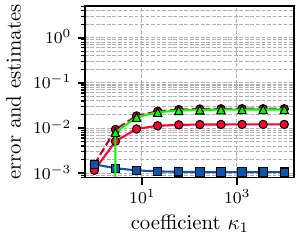}
    \caption{}
    \label{fig:numerical experiments:coeff:k1}
    \end{subfigure}\hfill
    \begin{subfigure}[T]{0.45\textwidth}
    \centering
    \includegraphics[width=0.8\linewidth]{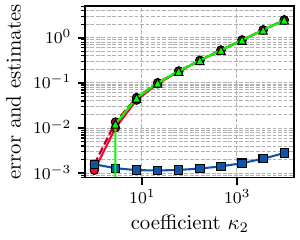}
    \caption{}
    \label{fig:numerical experiments:coeff:k2}
   
    \end{subfigure}
     \caption{T2: Errors and estimator components for varying contrasts in the cases (a) $\kappa_1\gg\kappa_2 = 1$ and (b) $\kappa_2\gg\kappa_1 = 1$.}
    \label{fig:numerical experiments:coeff}
\end{figure}

\begin{rem}
\label{rem:numerical experiments:asymptotic efficiency}
It is interesting to remark how the defeaturing component of the estimator
is asymptotically efficient in the limit $\rho \to \infty$. Let $\kappa_1 = 1$ and $\kappa_2 = \rho > 1$, so that the diffusion coefficient in the feature $\omega$ is equal to $\rho$ in the exact problem and equal to $1$ in the defeatured one.
Let
\begin{equation*}
  \widetilde{\mathcal E}_{\rm{d}}(u_0) := (\rho - 1)\,\kappa_2^{-1/2}\,\|\nabla u_0\|_\omega
\end{equation*}
denote the continuous modeling estimator, obtained by replacing the reconstructed flux $\flux$ in $\mathcal E_\mathrm{d}$ by the exact defeatured flux $-\kappa_0\nabla u_0$, which reduces to $-\nabla u_0$ on $\omega$.

Let us remark that subtracting problem \eqref{eq:model problem:weak form} and \eqref{eq:defeatured problem:weak form} leads to the orthogonality relation $a(u,v)-a_0(u_0,v)=0$ for any $v \in \honez{\Omega}$. It thus follows that
\begin{align*}
||e_\mathrm{d}||_a^2&=a(u-u_0,e_\mathrm{d})=a_0(u_0,e_\mathrm{d})-a(u_0,e_\mathrm{d})=(1-\rho)(\nabla u_0,\nabla e_\mathrm{d})_\patch\\
&=(\rho-1)\big(||\nabla u_0||_\patch^2 - (\nabla u_0,\nabla u)_\patch\big) = (\rho - 1)(1 - c)\,\|\nabla u_0\|_\omega^2
\end{align*}
with $c := \frac{(\nabla u_0, \nabla u)_\omega}{\|\nabla u_0\|_\omega^2}$
and therefore
\begin{equation*}
  \frac{\widetilde{\mathcal E}_{\rm{d}}(u_0)^2}{\|e_d\|_a^2} = \frac{1 - \rho^{-1}}{1 - c}.
\end{equation*}
It remains to show that $c \to 0$ as $\rho \to \infty$. Let us first of all remark that
\begin{align*}
\rho^{1/2}\|\nabla u\|_\omega \le \|u\|_a\leq C_P||f||_\Omega,
\end{align*}
using the stability estimate for the exact problem \eqref{eq:model problem:weak form}. Hence, $\|\nabla u\|_\omega = O(\rho^{-1/2})$. On the other hand, $u_0$ is only marginally influenced by the value of $\rho$, as it is computed considering a unitary diffusion coefficient within the feature $\patch$. This is also reflected by the behavior of $\|\flux \|_\omega$ being almost constant in \cref{fig:numerical experiments:coeff_scaling:k2}. By the Cauchy-Schwarz inequality,
\begin{align*}
  |c| \le \frac{\|\nabla u\|_\omega}{\|\nabla u_0\|_\omega} = O(\rho^{-1/2})
  \longrightarrow 0,
\end{align*}
\begin{figure}
	\centering
	\begin{subfigure}[T]{0.45\textwidth}
		\centering
		\includegraphics[width=0.7\linewidth]{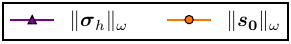}
	\end{subfigure}\vspace{0.1cm}\\
	\begin{subfigure}[T]{0.45\textwidth}
		\centering
		\includegraphics[width=0.75\linewidth]{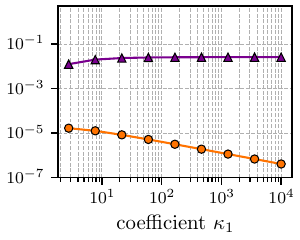}
		\caption{}
		\label{fig:numerical experiments:coeff_scaling:k1}
	\end{subfigure}\hfill
	\begin{subfigure}[T]{0.45\textwidth}
		\centering
		\includegraphics[width=0.75\linewidth]{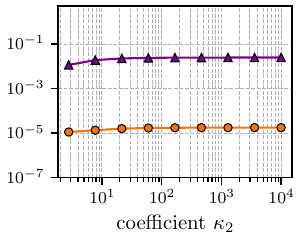}
		\caption{}
		\label{fig:numerical experiments:coeff_scaling:k2}
	\end{subfigure}
	\caption{T2: quantities involved in the assembly of the estimator for varying contrasts: (a) $\kappa_1\gg\kappa_2 = 1$ and (b) $\kappa_2\gg\kappa_1 = 1$.}
	\label{fig:numerical experiments:coeff_components}
\end{figure}
so that $\widetilde{\mathcal E}_{\rm{d}}(u_0)^2/\|e_d\|_a^2 \to 1$. Thus, we conclude that the continuous modeling estimator is asymptotically efficient as the contrast grows. Since $\mathcal E_{\rm{d}}(u_{0h}) \to \widetilde{\mathcal E}_{\rm{d}}(u_0)$ as $h \to 0$, the discrete estimator inherits this behavior up to the error in the flux reconstruction.
\end{rem}

\subsection{T3: Multiple features}
In this third numerical example, we aim to show how the estimator behaves in the presence of multiple features and how it can be used to assess the relevance of each feature at different mesh refinement levels.
We consider the geometry of \cref{fig:numerical experiments:geometries:multi}, which contains six features $\patch_i$, $1 \leq i \leq 6$; the pairs $(\patch_1,\patch_6)$, $(\patch_2,\patch_5)$, and $(\patch_3,\patch_4)$ have equal size but opposite orientation with respect to the defeatured interface $\Gamma_0$ given by the vertical centerline. We set $\kappa_1 = 1$ in $\domaincomp{1}$ and $\kappa_2 = 2$ in $\domaincomp{2}$, impose homogeneous Dirichlet boundary conditions on the outer boundary, and take a constant source term $f \equiv 1$.

\Cref{tab:numerical results:multifeature} reports the numerical component $\mathcal{E}_\mathrm{n}$ of the estimator together with the contribution $\mathcal{E}_\mathrm{d}^i$ of each feature $\patch_i$ to the defeaturing component $\mathcal{E}_\mathrm{d}$. Looking at the columns from 3 to 8, we observe how $\patch_6$ is the most relevant among the features, as $\estim{\mathrm{d}}^6>\estim{\mathrm{d}}^i$ for all $i \neq 6$ and for all the considered mesh sizes. However, it is important not only to provide a ranking of the features, but also to assess whether they are relevant for the overall accuracy of the solution. This is possible by comparing the values of $\estim{\mathrm{d}}^i$ with the value of the numerical component of the estimator $\estim{\mathrm{n}}$. For example, for the coarsest mesh considered ($h = 0.35$), $\patch_6$ is the only feature whose contribution to the estimator exceeds its numerical component. If a finer mesh is considered, the numerical source of error decreases, and a larger number of features impact the accuracy of the solution. On the intermediate mesh ($h = 0.075$), the two largest features $\patch_1$ and $\patch_6$ exceed the numerical contribution, and on the finest mesh, all six features do so.

The defeaturing error is asymmetric with respect to the direction of the material substitution. In particular, features that replace a high-diffusivity inclusion with the lower-diffusivity background ($\patch_2$, $\patch_4$, $\patch_6$) yield defeaturing contributions $1.5$ to $2$ times larger than the symmetric features ($\patch_1$, $\patch_3$, $\patch_5$), which replace a low-diffusivity inclusion with the higher-diffusivity background. This again highlights the solution's greater sensitivity to the omission of highly conductive regions compared to resistive ones, in accordance with experiment~T2. A slight sensitivity of the values of $\estim{\mathrm{d}}^i$ on the mesh size depends on the fact that the considered meshes are not nested.

Overall, this experiment shows how the estimator can be used as a tool to localize the modeling error to individual features, allowing the contribution of each feature to be quantified separately and the most relevant features to be identified at a given mesh size.

\sisetup{
  round-mode = figures,
  round-precision = 3,
  scientific-notation = true,
  tight-spacing = true,
  output-exponent-marker = {e},
}
\begin{table}[t]
\centering
\small 
\setlength{\tabcolsep}{2pt} 
\caption{T3: components of the total estimator for different mesh sizes.}
\label{tab:numerical results:multifeature}
\begin{tabular}{@{}S[table-format=1.2e-1]@{\hspace{6pt}}|@{\hspace{6pt}}S[table-format=1.2e-2]|*{6}{S[table-format=1.2e-2]}@{}}
\toprule
{mesh size $h$} & {$\mathcal{E}_\mathrm{n}(u_{0h})$} & {$\estim{\mathrm{d}}^1$} & {$\estim{\mathrm{d}}^2$} & {$\estim{\mathrm{d}}^3$} & {$\estim{\mathrm{d}}^4$} & {$\estim{\mathrm{d}}^5$} & {$\estim{\mathrm{d}}^6$} \\ \midrule
3.5e-1 & 0.0568827639633597 & 0.0412155994754566 & 0.0112950182542503 & 0.0137751286425658 & 0.0180717552070865 & 0.00730654841368795 & 0.0701134327294442 \\
7.5e-2 & 0.0226821868793000 & 0.0356090676337706 & 0.0102993045642681 & 0.0107252223710728 & 0.0137733593812597 & 0.00607748822390633 & 0.0659412061169828 \\
1e-2   & 0.00315268602575921 & 0.0329437836229053 & 0.0113661417698191 & 0.00899711612678531 & 0.0138036851635137 & 0.00551514277253549 & 0.0705829916548984 \\ \bottomrule
\end{tabular}
\end{table}

\subsection{T4: Inclusion removal}
In this last numerical experiment, we consider the scenario in which $\Gamma_0=\emptyset$. This corresponds to removing an entire 
material inclusion whose diffusion coefficient differs from that of the ambient domain. As observed in \cref{rem:estimator:inclusions}, this scenario is not excluded by our theory, and this test aims to showcase this.  Practically speaking, this could be the case, for example, of a screw in a bidimensional section of a mechanical component, or of some impurity embedded in a material.

In our test case, we consider a square $\Omega=(-1,1)^2$, such that $\Omega=\mathrm{int}(\overline{\Omega_1}\cup\overline{\Omega_2})$, with $\Omega_2=(-\epsilon/2, \epsilon/2)^2$ and $\epsilon<2$ (see \cref{fig:numerical experiments:inclusion}). With respect to the previously introduced notation, the entire $\Omega_2$ represents our feature $\patch_\epsilon$. We set $\kappa_1 = 10$ in $\Omega_1=\Omega \setminus \overline{\Omega_2}$ whereas $\kappa_2 = 1$ in $\Omega_2$. In the defeatured problem, we set $\kappa_0=\kappa_1$ on the whole of $\Omega$.

\Cref{fig:numerical experiments:inclusion:estimate} shows the trend of the energy norm of the error, of the estimator, and of its components for the feature side $\epsilon$ varying in the interval $[0.01, 0.5]$.
The behavior matches that of the feature-size study in T1: the defeaturing component decreases linearly in $\epsilon$, while the numerical component remains approximately constant. With respect to the given mesh size ($h=0.0075$), the inclusion appears to have a significant impact on the solution accuracy for $\epsilon>5\cdot 10^{-2}$. For lower feature sizes, the error and the estimator reach a plateau, both steered by their numerical component.

Let us remark how, in the limit of $\kappa_2\rightarrow 0$, the considered setting approximates the one considered in \cite{buffa_analysis-aware_2022, buffa_equilibrated_2024}, where the impact of filling internal holes with material is considered.

\begin{figure}
    \centering
    \begin{subfigure}[T]{0.33\textwidth}
    \centering\vspace{0.5cm}
    \begin{tikzpicture}[scale=0.6]
    \def\outer{5}
    \def\inner{2.5}
    \def\featureWidth{.6}
    \filldraw[thick, fill=gray!15] (-\outer / 2, -\outer / 2) rectangle (\outer / 2, \outer / 2);
    \filldraw[thick, blue1, fill=white] (-\inner / 2, -\inner / 2) rectangle (\inner / 2, \inner / 2);

    \node[anchor=south west] at (-\outer / 2, -\outer / 2) {$\domaincomp{2}$};
    \node at (0, 0) {$\defdomaincomp{1}=\patch$};

    \node[blue1, anchor=west] at (\inner / 2, 0) {$\defbd$};
    \node[anchor=north] at (0, -\inner / 2) {$s$};
    
\end{tikzpicture}\vspace{0.5cm}
    \caption{}
    \label{fig:numerical experiments:geometry}
    \end{subfigure}
    \hfill
    \begin{subfigure}[T]{0.66\textwidth}
    \centering
    \includegraphics{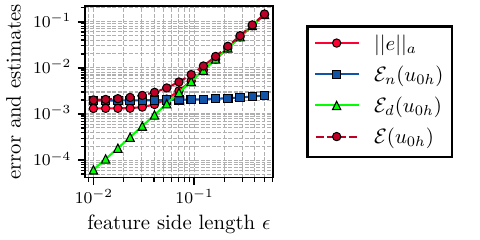}
    \caption{}
    \label{fig:numerical experiments:inclusion:estimate}
    \end{subfigure}\vspace{-0.5cm}
    \caption{T4: Illustration of the square geometry with inclusion (a) and the corresponding defeaturing errors and estimates (b). }
    \label{fig:numerical experiments:inclusion}
\end{figure}
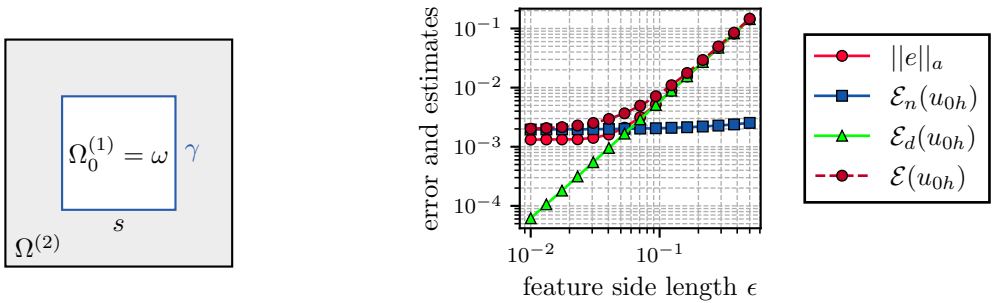

\section{Conclusion}
\label{sec:conclusion}
In this work, we have proposed a combined \emph{a posteriori} error estimator capable of assessing the impact of interface coarsening on the accuracy of the solution for elliptic problems characterized by internal interfaces. The proposed error estimator is based on an equilibrated flux reconstructed from the coarsened-interface numerical solution on a mesh that conforms to the coarsened interface but can be blind to the original one. The difference between the numerical and the equilibrated flux is used to bound the numerical source of error. The modeling component of the error, deriving from the interface coarsening, is instead bounded by a computable approximation of the $L^2$-norm of the reconstructed flux over the coarsened feature. By construction, it decomposes into independent contributions from each feature, providing a feature-by-feature estimate of the modeling error.

We proved that this estimator is reliable for a single feature and extended the result to multiple features and to the removal of entire material inclusions.
The numerical experiments confirm the expected behavior under mesh refinement and feature-size reduction. The application in the multiple-feature case provides a rationale for ranking the features by their impact and assessing their relevance to the overall accuracy of the simulation at a fixed mesh size. The estimator also remains reliable across a broad range of contrasts. When the conducting subdomain is simplified, the modeling estimator is asymptotically efficient as $\rho \to \infty$.

Several research directions remain open. A combined mesh and geometric adaptive strategy, in the spirit of \cite{buffa_adaptive_2024} and \cite{buffa_adaptive_2025}, would exploit the localized modeling estimator to drive both feature inclusion and local refinement. In that case, allowing for an unfitted flux reconstruction would facilitate the inclusion of features deemed relevant by the estimator, without the need to modify the mesh built on the defeatured geometry. Finally, the framework could be extended to more complex transmission problems or to goal-oriented estimators, measuring the impact of interface coarsening on specific quantities of interest rather than only on the solution's accuracy in the global energy norm.

\section*{Acknowledgments}
The authors wish to thank Prof. Annalisa Buffa (EPFL, Switzerland) for many fruitful discussions and for her thoughtful advice. 

Author Philipp Weder was supported by the Swiss National Science Foundation via project MINT n. 200021\_215099, PDE tools for analysis-aware geometry processing in simulation science.

\printbibliography

@article{antolin_analysisaware_2024,
	title = {Analysis‐aware defeaturing of complex geometries with {Neumann} features},
	volume = {125},
	issn = {0029-5981, 1097-0207},
	url = {https://onlinelibrary.wiley.com/doi/10.1002/nme.7380},
	doi = {10.1002/nme.7380},
	language = {en},
	number = {3},
	journal = {International Journal for Numerical Methods in Engineering},
	author = {Antolín, Pablo and Chanon, Ondine},
	month = feb,
	year = {2024},
	pages = {e7380},
}

@article{buffa_equilibrated_2024,
	title = {An {Equilibrated} {Flux} {A} {Posteriori} {Error} {Estimator} for {Defeaturing} {Problems}},
	volume = {62},
	issn = {0036-1429, 1095-7170},
	url = {https://epubs.siam.org/doi/10.1137/23M1627195},
	doi = {10.1137/23M1627195},
	language = {en},
	number = {6},
	journal = {SIAM Journal on Numerical Analysis},
	author = {Buffa, Annalisa and Chanon, Ondine and Grappein, Denise and Vázquez, Rafael and Vohralík, Martin},
	month = dec,
	year = {2024},
	pages = {2439--2458},
}

@article{buffa_analysis-aware_2022,
	title = {Analysis-aware defeaturing: {Problem} setting and \textit{a posteriori} estimation},
	volume = {32},
	issn = {0218-2025, 1793-6314},
	shorttitle = {Analysis-aware defeaturing},
	url = {https://www.worldscientific.com/doi/10.1142/S0218202522500099},
	doi = {10.1142/S0218202522500099},
	language = {en},
	number = {02},
	journal = {Mathematical Models and Methods in Applied Sciences},
	author = {Buffa, Annalisa and Chanon, Ondine and Vázquez, Rafael},
	month = feb,
	year = {2022},
	pages = {359--402},
}

@misc{buffa_adaptive_2024,
	title = {Adaptive analysis-aware defeaturing: the case of {Neumann} boundary conditions},
	shorttitle = {Adaptive analysis-aware defeaturing},
	url = {http://arxiv.org/abs/2212.05183},
	doi = {10.48550/arXiv.2212.05183},
	publisher = {arXiv},
	author = {Buffa, Annalisa and Chanon, Ondine and Vázquez, Rafael},
	month = jul,
	year = {2024},
	note = {arXiv:2212.05183 [math]},
}

@misc{buffa_adaptive_2025,
	title = {Adaptive refinement in defeaturing problems via an equilibrated flux a posteriori error estimator},
	url = {http://arxiv.org/abs/2503.19784},
	doi = {10.48550/arXiv.2503.19784},
	language = {en},
	publisher = {arXiv},
	author = {Buffa, Annalisa and Grappein, Denise and Vázquez, Rafael},
	month = mar,
	year = {2025},
	note = {arXiv:2503.19784 [math]},
}

@misc{weder_analysis-aware_2025,
	title = {Analysis-{Aware} {Defeaturing} of {Dirichlet} {Features}},
	copyright = {All rights reserved},
	url = {http://arxiv.org/abs/2508.13886},
	doi = {10.48550/arXiv.2508.13886},
	publisher = {arXiv},
	author = {Weder, Philipp and Buffa, Annalisa},
	month = aug,
	year = {2025},
	note = {arXiv:2508.13886 [math]},
}

@book{ern_finite_2021,
	address = {Cham},
	series = {Texts in {Applied} {Mathematics}},
	title = {Finite {Elements} {II}: {Galerkin} {Approximation}, {Elliptic} and {Mixed} {PDEs}},
	volume = {73},
	copyright = {https://www.springer.com/tdm},
	isbn = {978-3-030-56922-8 978-3-030-56923-5},
	shorttitle = {Finite {Elements} {II}},
	url = {https://link.springer.com/10.1007/978-3-030-56923-5},
	doi = {10.1007/978-3-030-56923-5},
	language = {en},
	publisher = {Springer International Publishing},
	author = {Ern, Alexandre and Guermond, Jean-Luc},
	year = {2021},
}

@article{ern_polynomial-degree-robust_2015,
	title = {Polynomial-{Degree}-{Robust} {A} {Posteriori} {Estimates} in a {Unified} {Setting} for {Conforming}, {Nonconforming}, {Discontinuous} {Galerkin}, and {Mixed} {Discretizations}},
	volume = {53},
	issn = {0036-1429, 1095-7170},
	url = {http://epubs.siam.org/doi/10.1137/130950100},
	doi = {10.1137/130950100},
	language = {en},
	number = {2},
	journal = {SIAM Journal on Numerical Analysis},
	author = {Ern, Alexandre and Vohralík, Martin},
	month = jan,
	year = {2015},
	pages = {1058--1081},
}

@misc{weder_certified_2025,
	title = {A {Certified} {Goal}-{Oriented} {A} {Posteriori} {Defeaturing} {Error} {Estimator} for {Elliptic} {PDEs}},
	url = {http://arxiv.org/abs/2512.20124},
	doi = {10.48550/arXiv.2512.20124},
	language = {en},
	publisher = {arXiv},
	author = {Weder, Philipp and Buffa, Annalisa},
	month = dec,
	year = {2025},
	note = {arXiv:2512.20124 [math]},
}

@article{braess_equilibrated_2007,
	title = {Equilibrated residual error estimator for edge elements},
	volume = {77},
	issn = {0025-5718},
	url = {http://www.ams.org/journal-getitem?pii=S0025-5718-07-02080-7},
	doi = {10.1090/S0025-5718-07-02080-7},
	language = {en},
	number = {262},
	journal = {Mathematics of Computation},
	author = {Braess, Dietrich and Schöberl, Joachim},
	month = nov,
	year = {2007},
	pages = {651--673},
}

@book{braess_finite_2007,
	edition = {3},
	title = {Finite {Elements}: {Theory}, {Fast} {Solvers}, and {Applications} in {Solid} {Mechanics}},
	copyright = {https://www.cambridge.org/core/terms},
	isbn = {978-0-521-70518-9 978-0-511-61863-5},
	shorttitle = {Finite {Elements}},
	url = {https://www.cambridge.org/core/product/identifier/9780511618635/type/book},
	doi = {10.1017/CBO9780511618635},
	language = {en},
	publisher = {Cambridge University Press},
	author = {Braess, Dietrich},
	month = apr,
	year = {2007},
}

@article{repin_combined_2012,
	title = {Combined \textit{a posteriori} modeling-discretization error estimate for elliptic problems with complicated interfaces},
	volume = {46},
	issn = {0764-583X, 1290-3841},
	url = {http://www.esaim-m2an.org/10.1051/m2an/2012007},
	doi = {10.1051/m2an/2012007},
	language = {en},
	number = {6},
	journal = {ESAIM: Mathematical Modelling and Numerical Analysis},
	author = {Repin, Sergey I. and Samrowski, Tatiana S. and Sauter, Stéfan A.},
	month = nov,
	year = {2012},
	pages = {1389--1405},
}

@misc{capatina_elliptic_2025,
	title = {Elliptic interface problem approximated by {CutFEM}: {I}. {Conservative} flux recovery and numerical validation of adaptive mesh refinement},
	shorttitle = {Elliptic interface problem approximated by {CutFEM}},
	url = {http://arxiv.org/abs/2507.03492},
	doi = {10.48550/arXiv.2507.03492},
	language = {en},
	publisher = {arXiv},
	author = {Capatina, Daniela and Gouasmi, Aimene and He, Cuiyu},
	month = jul,
	year = {2025},
	note = {arXiv:2507.03492 [math]},
}

@article{geuzaine_gmsh_2009,
	title = {Gmsh: {A} 3‐{D} finite element mesh generator with built‐in pre‐ and post‐processing facilities},
	volume = {79},
	copyright = {http://onlinelibrary.wiley.com/termsAndConditions\#vor},
	issn = {0029-5981, 1097-0207},
	shorttitle = {Gmsh},
	url = {https://onlinelibrary.wiley.com/doi/10.1002/nme.2579},
	doi = {10.1002/nme.2579},
	language = {en},
	number = {11},
	journal = {International Journal for Numerical Methods in Engineering},
	author = {Geuzaine, Christophe and Remacle, Jean‐François},
	month = sep,
	year = {2009},
	pages = {1309--1331},
}

@misc{baratta_dolfinx_2023,
	title = {{DOLFINx}: {The} next generation {FEniCS} problem solving environment},
	copyright = {Creative Commons Attribution 4.0 International},
	shorttitle = {{DOLFINx}},
	url = {https://zenodo.org/doi/10.5281/zenodo.10447666},
	doi = {10.5281/ZENODO.10447666},
	language = {en},
	publisher = {Zenodo},
	author = {Baratta, Igor A. and Dean, Joseph P. and Dokken, Jørgen S. and Habera, Michal and Hale, Jack S. and Richardson, Chris N. and Rognes, Marie E. and Scroggs, Matthew W. and Sime, Nathan and Wells, Garth N.},
	month = dec,
	year = {2023},
}

@article{prager_approximations_1947,
	title = {Approximations in elasticity based on the concept of function space},
	volume = {5},
	issn = {0033-569X, 1552-4485},
	url = {https://www.ams.org/qam/1947-05-03/S0033-569X-1947-25902-8/},
	doi = {10.1090/qam/25902},
	language = {en},
	number = {3},
	journal = {Quarterly of Applied Mathematics},
	author = {Prager, W. and Synge, J. L.},
	month = oct,
	year = {1947},
	pages = {241--269},
}

@article{fine_automated_2000,
	title = {Automated generation of {FEA} models through idealization operators},
	volume = {49},
	copyright = {http://doi.wiley.com/10.1002/tdm\_license\_1.1},
	issn = {0029-5981, 1097-0207},
	url = {https://onlinelibrary.wiley.com/doi/10.1002/1097-0207(20000910/20)49:1/2<83::AID-NME924>3.0.CO;2-N},
	doi = {10.1002/1097-0207(20000910/20)49:1/2<83::AID-NME924>3.0.CO;2-N},
	language = {en},
	number = {1-2},
	journal = {International Journal for Numerical Methods in Engineering},
	author = {Fine, L. and Remondini, L. and Leon, J.-C.},
	month = sep,
	year = {2000},
	pages = {83--108},
}

@article{foucault_mechanical_2004,
	title = {Mechanical {Criteria} for the {Preparation} of {Finite} {Element} {Models}.},
	journal = {Proceedings of the 13th International Meshing Roundtable},
	author = {Foucault, Gilles and Marin, Philippe and Léon, Jean-Claude},
	year = {2004},
	pages = {413--426},
}

@article{thakur_survey_2009,
	title = {A survey of {CAD} model simplification techniques for physics-based simulation applications},
	volume = {41},
	copyright = {https://www.elsevier.com/tdm/userlicense/1.0/},
	issn = {00104485},
	url = {https://linkinghub.elsevier.com/retrieve/pii/S0010448508002285},
	doi = {10.1016/j.cad.2008.11.009},
	language = {en},
	number = {2},
	journal = {Computer-Aided Design},
	author = {Thakur, Atul and Banerjee, Ashis Gopal and Gupta, Satyandra K.},
	month = feb,
	year = {2009},
	pages = {65--80},
}

@book{bensoussan_asymptotic_2011,
	address = {Providence, R.I},
	series = {{AMS} {Chelsea} {Publishing}},
	title = {Asymptotic analysis for periodic structures},
	isbn = {978-0-8218-5324-5 978-1-4704-1581-5},
	language = {eng},
	number = {v. 374},
	publisher = {American Mathematical Society},
	author = {Bensoussan, Alain},
	collaborator = {Lions, Jacques-Louis and Papanicolaou, George},
	year = {2011},
}

@book{bakhvalov_homogenisation_1989,
	address = {Dordrecht},
	series = {Mathematics and its {Applications}},
	title = {Homogenisation: {Averaging} {Processes} in {Periodic} {Media}},
	volume = {36},
	copyright = {http://www.springer.com/tdm},
	isbn = {978-94-010-7506-0 978-94-009-2247-1},
	shorttitle = {Homogenisation},
	url = {http://link.springer.com/10.1007/978-94-009-2247-1},
	publisher = {Springer Netherlands},
	author = {Bakhvalov, N. and Panasenko, G.},
	editor = {Hazewinkel, M.},
	year = {1989},
	doi = {10.1007/978-94-009-2247-1},
}

@article{peskin_immersed_2002,
	title = {The immersed boundary method},
	volume = {11},
	copyright = {https://www.cambridge.org/core/terms},
	issn = {0962-4929, 1474-0508},
	url = {https://www.cambridge.org/core/product/identifier/S0962492902000077/type/journal_article},
	doi = {10.1017/S0962492902000077},
	language = {en},
	journal = {Acta Numerica},
	author = {Peskin, Charles S.},
	month = jan,
	year = {2002},
	pages = {479--517},
}

@article{fries_extendedgeneralized_2010,
	title = {The extended/generalized finite element method: {An} overview of the method and its applications},
	volume = {84},
	copyright = {http://onlinelibrary.wiley.com/termsAndConditions\#vor},
	issn = {0029-5981, 1097-0207},
	shorttitle = {The extended/generalized finite element method},
	url = {https://onlinelibrary.wiley.com/doi/10.1002/nme.2914},
	doi = {10.1002/nme.2914},
	language = {en},
	number = {3},
	journal = {International Journal for Numerical Methods in Engineering},
	author = {Fries, Thomas‐Peter and Belytschko, Ted},
	month = oct,
	year = {2010},
	pages = {253--304},
}

@article{burman_fictitious_2012,
	title = {Fictitious domain finite element methods using cut elements: {II}. {A} stabilized {Nitsche} method},
	volume = {62},
	copyright = {https://www.elsevier.com/tdm/userlicense/1.0/},
	issn = {01689274},
	shorttitle = {Fictitious domain finite element methods using cut elements},
	url = {https://linkinghub.elsevier.com/retrieve/pii/S0168927411000298},
	doi = {10.1016/j.apnum.2011.01.008},
	language = {en},
	number = {4},
	journal = {Applied Numerical Mathematics},
	author = {Burman, Erik and Hansbo, Peter},
	month = apr,
	year = {2012},
	pages = {328--341},
}

@article{burman_fictitious_2010,
	title = {Fictitious domain finite element methods using cut elements: {I}. {A} stabilized {Lagrange} multiplier method},
	volume = {199},
	copyright = {https://www.elsevier.com/tdm/userlicense/1.0/},
	issn = {00457825},
	shorttitle = {Fictitious domain finite element methods using cut elements},
	url = {https://linkinghub.elsevier.com/retrieve/pii/S004578251000160X},
	doi = {10.1016/j.cma.2010.05.011},
	language = {en},
	number = {41-44},
	journal = {Computer Methods in Applied Mechanics and Engineering},
	author = {Burman, Erik and Hansbo, Peter},
	month = oct,
	year = {2010},
	pages = {2680--2686},
}

@article{boffi_finite_2003,
	title = {A finite element approach for the immersed boundary method},
	volume = {81},
	copyright = {https://www.elsevier.com/tdm/userlicense/1.0/},
	issn = {00457949},
	url = {https://linkinghub.elsevier.com/retrieve/pii/S0045794902004042},
	doi = {10.1016/S0045-7949(02)00404-2},
	language = {en},
	number = {8-11},
	journal = {Computers \& Structures},
	author = {Boffi, Daniele and Gastaldi, Lucia},
	month = may,
	year = {2003},
	pages = {491--501},
}

@article{borgers_finite_1990,
	title = {On {Finite} {Element} {Domain} {Imbedding} {Methods}},
	volume = {27},
	issn = {0036-1429, 1095-7170},
	url = {http://epubs.siam.org/doi/10.1137/0727055},
	doi = {10.1137/0727055},
	language = {en},
	number = {4},
	journal = {SIAM Journal on Numerical Analysis},
	author = {Börgers, Christoph and Widlund, Olof B.},
	month = aug,
	year = {1990},
	pages = {963--978},
}

@article{babuska_finite_1973,
	title = {The {Finite} {Element} {Method} with {Penalty}},
	volume = {27},
	issn = {00255718},
	url = {https://www.jstor.org/stable/2005611?origin=crossref},
	doi = {10.2307/2005611},
	number = {122},
	journal = {Mathematics of Computation},
	author = {Babuska, Ivo},
	month = apr,
	year = {1973},
	pages = {221},
}

@article{weymuth_posteriori_2017,
	title = {A {Posteriori} {Modelling}-{Discretization} {Error} {Estimate} for {Elliptic} {Problems} with \textit{{L}}$^{\infty}$ -{Coefficients}},
	volume = {17},
	issn = {1609-9389, 1609-4840},
	url = {https://www.degruyter.com/document/doi/10.1515/cmam-2017-0015/html},
	doi = {10.1515/cmam-2017-0015},
	language = {en},
	number = {3},
	journal = {Computational Methods in Applied Mathematics},
	author = {Weymuth, Monika and Sauter, Stefan and Repin, Sergey},
	month = jul,
	year = {2017},
	pages = {515--531},
}

@article{capatina_elliptic_2026,
	title = {Elliptic interface problem approximated by {CutFEM}: {II}. {A} posteriori error analysis based on equilibrated fluxes},
	copyright = {https://www.edpsciences.org/en/authors/copyright-and-licensing},
	issn = {2822-7840, 2804-7214},
	shorttitle = {Elliptic interface problem approximated by {CutFEM}},
	url = {https://www.esaim-m2an.org/10.1051/m2an/2026043},
	doi = {10.1051/m2an/2026043},
	journal = {ESAIM: Mathematical Modelling and Numerical Analysis},
	author = {Capatina, Daniela and Gouasmi, Aimene},
	month = may,
	year = {2026},
}

\end{document}